\documentclass{amsart}
\usepackage{amsmath}
\usepackage{amssymb}
\usepackage[dvips]{graphicx}
\usepackage[latin1]{inputenc}
\usepackage{amssymb,latexsym}
\newcommand{\R}{\mathbb R}
\newcommand{\Z}{\mathbb Z}

\newcommand{\C}{\mathbb C}

\newcommand{\iz}{\mathbf{i}}

\newtheorem{theorem}{Theorem}
\newtheorem{lemma}{Lemma}

\begin{document}

\title[Extremal first Dirichlet eigenvalue ]{Extremal first Dirichlet eigenvalue of doubly connected plane domains and dihedral symmetry}

\author{ Ahmad El Soufi and Rola Kiwan }

\address{ Laboratoire de Math\'ematiques et Physique Th\'eorique,
UMR CNRS 6083, Universit\'e Fran\c{c}ois Rabelais de Tours, Parc de Grandmont, F-37200
Tours France}
\email{elsoufi@univ-tours.fr ; kiwan@lmpt.univ-tours.fr}



\begin{abstract}
We deal with the following eigenvalue optimization problem: 
Given a bounded domain $D\subset \R^2$, how to place an obstacle $B$ of fixed shape within $D$ so as to maximize or minimize the fundamental eigenvalue $\lambda_1$ of the Dirichlet Laplacian on $D\setminus B$. This means that we want to extremize the function $\rho\mapsto \lambda_1(D\setminus \rho (B))$, where $\rho$ runs over the set of rigid motions such that $\rho (B)\subset D$. 
We answer this problem in the case where both $D$ and $B$ are invariant under the action of a dihedral group $\mathbb{D}_n$, $n\ge2$, and where the distance from the origin to the boundary is monotonous as a function of the argument between two axes of symmetry. The extremal configurations correspond to the cases where the axes of symmetry of $B$ coincide with those of $D$.

\end{abstract}

\maketitle

\section {Introduction and Statement of the main Result}\label{1}
The relations between the shape of a domain and the eigenvalues of its Dirichlet or Neumann Laplacian,  have been intensively investigated since the 1920's when Faber \cite{F} and Krahn \cite{Kr} have proved independently the famous eigenvalue isoperimetric inequality first conjectured by Rayleigh (1877): the first Dirichlet eigenvalue $\lambda_1(\Omega)$ of any bounded domain $\Omega\subset\R^n$ satisfies
$$\lambda_1(\Omega)\ge \lambda_1(\Omega^*),$$
where $\Omega^*$ is a ball having the same volume as $\Omega$. We refer to the review papers of Ashbaugh \cite{A1,A2} and Henrot \cite{He} for a survey of recent results on optimization problems involving eigenvalues.


The present work deals with the following eigenvalue optimization problem: 
Given a bounded domain $D$, we want to place an obstacle (or a hole) $B$, of fixed shape, inside $D$ so as to maximize or minimize the fundamental eigenvalue $\lambda_1$ of the Laplacian or Schr\"odinger operator on $D\setminus B$ with Zero Dirichlet conditions on the boundary.

In other words, the problem is to optimize the principal eigenvalue function $\rho\mapsto \lambda_1(D\setminus \rho (B))$, where $\rho$ runs over the set of rigid motions such that $\rho (B)\subset D$.

The first result obtained in this direction concerned the case where both $D$ and $B$ are disks of given radii. Indeed, it follows from Hersch's work \cite{H2} that the maximum of $\lambda_1$ is achieved when the disks are concentric (see also \cite{RS}). This result has been extended to any dimension by several authors (Harrell, Kr\"oger and Kurata \cite{HKK}, Kesavan \cite{K}, ...). Actually, Harrell, Kr\"oger and Kurata \cite{HKK} gave a more general result showing that, if the domain $D$ satisfies an interior symmetry property with respect to a hyperplane $P$ passing through the center of the spherical obstacle $B$ (which means that the image by the reflection with respect to $P$ of one component of $D\setminus P$ is contained in $D$), then the Dirichlet fundamental eigenvalue $\lambda_1(D\setminus B)$ decreases when the center of $B$ moves perpendicularly to $P$ in the direction of the boundary of $D$. In the particular case where both the domain $D$ and the obstacle $B$ are balls, this implies that the minimum of $\lambda_1(D\setminus B)$ corresponds to the limit case where $B$ touches the boundary of $D$.
 
Notice that when the obstacle $B$ is a disk, only translations of $B$ may affect the $\lambda_1$ of  $D\setminus B$ and the optimal placement problem reduces to the choice of the center of $B$ inside $D$.

In the present work we investigate a kind of dual problem in the sense that we consider a \emph{nonspherical} obstacle $B$ whose center of mass is fixed inside $D$, and seek the optimal positions while turning $B$ around its center. 

It is of course hopeless to expect a universal solution to this problem. In fact, we will restrict our investigation to a class of domains satisfying a dihedral symmetry and a monotonicity conditions.

 Thus, let $D$ be a simply-connected plane domain and assume that the following conditions are satisfied:

(i) ($\mathbb{D}_n$-symmetry) for an integer $n\ge 2$, $D$ is invariant under the action of the dihedral group $\mathbb{D}_n$ of order $2n$ generated by the rotation
$\rho_{\frac{2\pi}n}$ of angle $\frac{2\pi}n$ and a reflection $S$. 
Such a domain admits $n$ axes of symmetry passing through the origin and such that the angle between 2 consecutive axes is $\frac\pi n$. 

(ii) (monotonicity of the boundary) the distance $d(O,x)$ from the origin to a point $x$ of the boundary of $D$ is monotonous as a function of the argument of $x$,
 in a sector delimited by two consecutive symmetry axes.

Notice that assumption (i) guarantees that the center of mass of $D$ is at the origin. Regular $n$-gones centered at the origin are the simplest examples of domains satisfying these assumptions. More generally, if $g$ is any positive even $\frac{2 \pi}n$-periodic continuous function that is monotonous on the interval $(0,\frac\pi n)$, then the domain
$$D=\{re^{\iz\theta}; \theta\in [0,2\pi),0\leq r< g(\theta)\},$$
satisfies assumptions (i) and (ii). Actually, up to a rigid motion, any domain satisfying assumptions (i) and (ii) can be parametrized in such a manner.

It is worth noticing that, due to the monotonicity condition, the ``distance to the origin" function on the boundary of $D$ achieves its maximum and its minimum alternatively at the intersection points of $\partial D$ with the $2n$ half-axes of symmetry. The $n$ points of $\partial D$ at maximal (resp. minimal) distance from the origin will be called "outer vertices" (resp. "inner vertices") of D.   

Our main result is the following

\begin{theorem}\label{main}  
Let $D$ and $B$ be two plane domains satisfying the assumptions of ${\mathbb D}_n$-symmetry and monotonicity (i) and (ii) above for an  integer $n\ge2$. Assume furthermore that $B$ has $C^2$ boundary and that $\rho (B)\subset D$ for all $\rho\in SO(2)$. Then, the fundamental Dirichlet eigenvalue $\lambda_1(D\setminus B)$ of $D\setminus B$ is optimized exactly when the axes of symmetry of $B$ coincide with those of $D$. 

The maximizing configuration corresponds to the case where the outer vertices of $B$ and $D$ lie on the same half-axes of symmetry (we will then say that $B$ occupies the ``ON" position in $D$).  

The minimizing configuration corresponds to the case where the outer vertices of $B$ lie on the half-axes of symmetry passing through the inner vertices of $D$ (this is what will be called the ``OFF" position).  

\end{theorem}

Actually, we will prove that, except for the trivial case where $D$ or $B$ is a disk, the fundamental Dirichlet eigenvalue of $D\setminus B$ decreases gradually when $B$ switches from  ``ON" to ``OFF". 

The main ingredients of the proof of Theorem \ref{main} are Hadamard's variation formula for $\lambda_1$ and the technique of domain reflection initiated by Serrin \cite{Serrin} in PDE's setting.

\begin{center}
\includegraphics[angle=0,width=5cm]{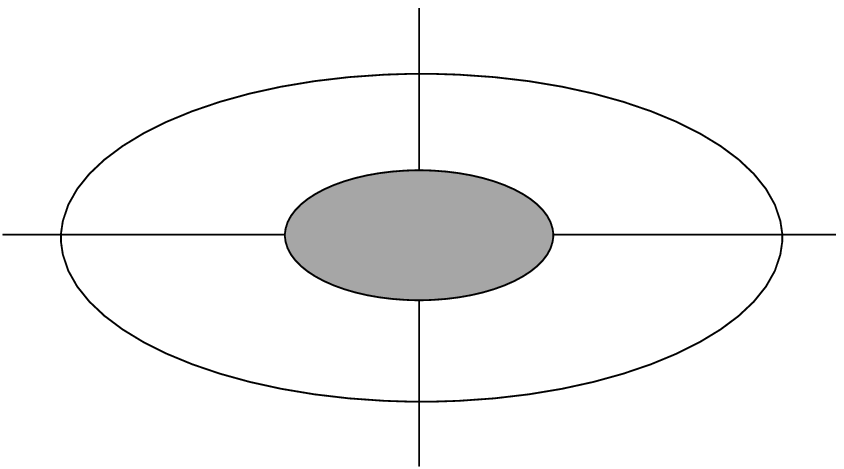} \hspace{1cm}
\includegraphics[angle=0,width=5cm]{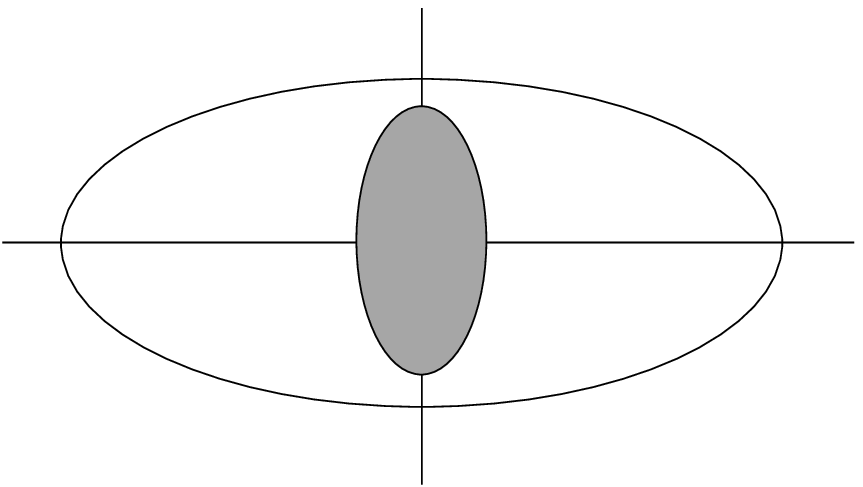}
\end{center}

\begin{center}
\includegraphics[angle=0,width=5cm]{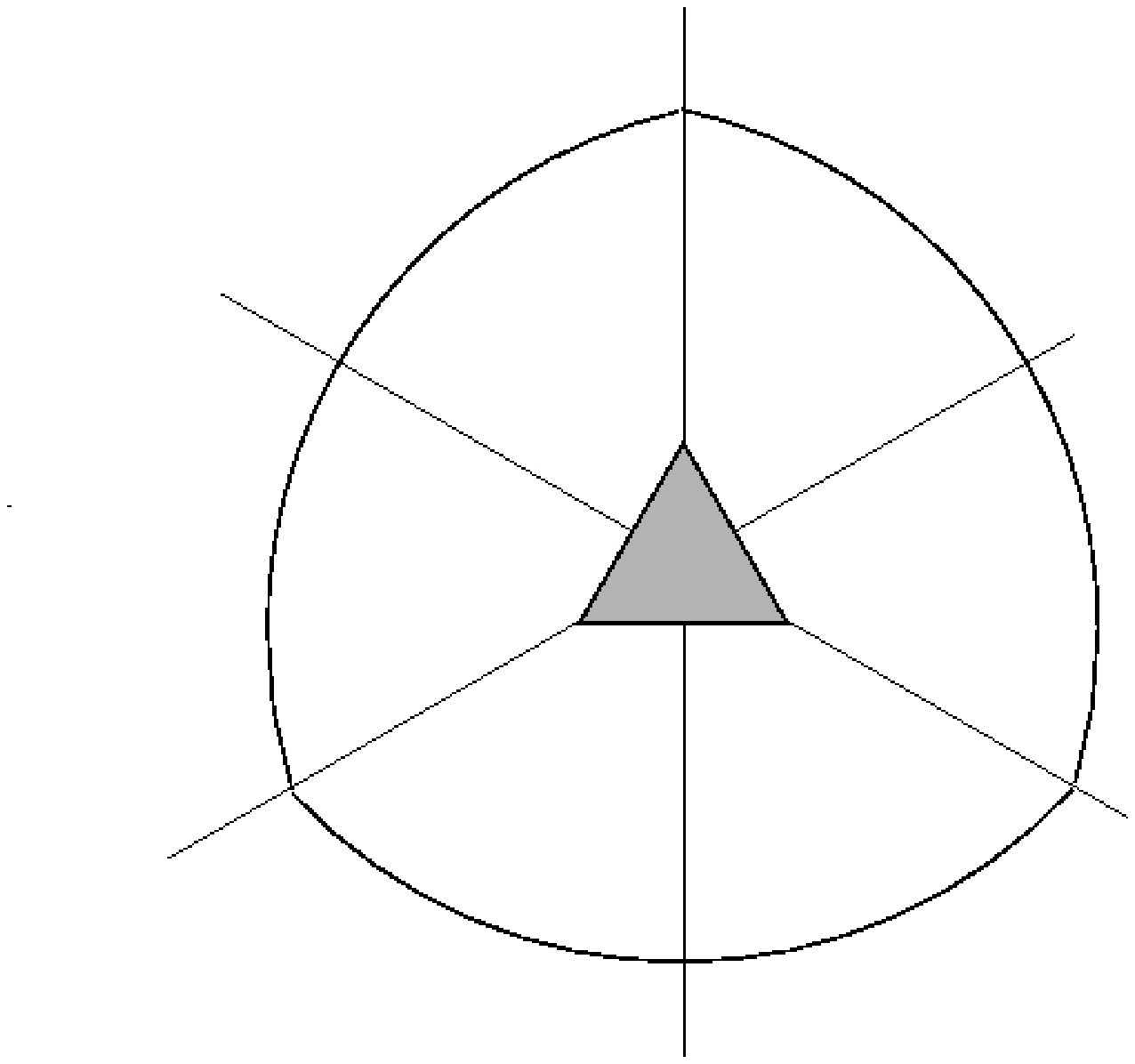}  \hspace{1cm}
\includegraphics[angle=0,width=5cm]{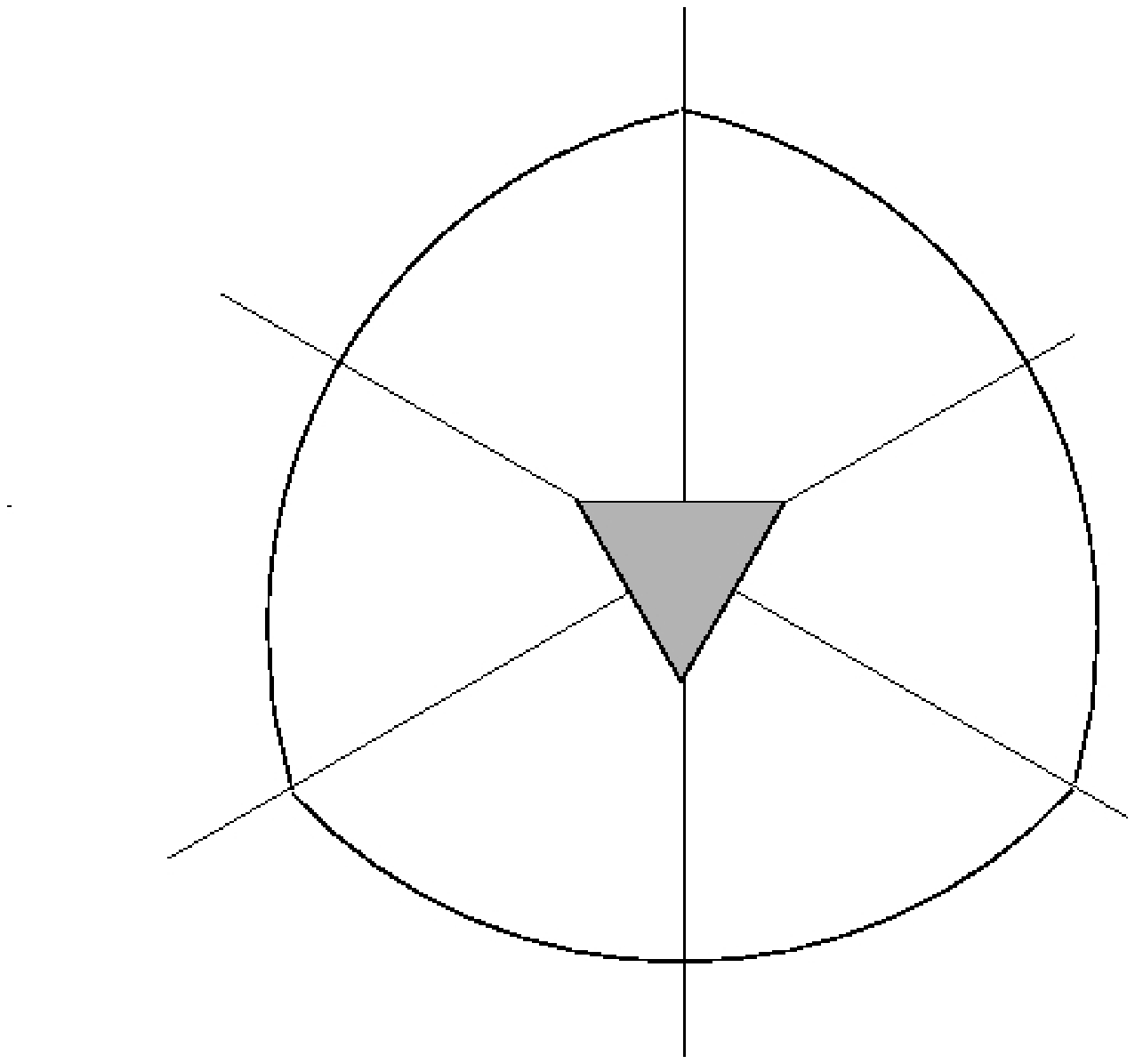}
\end{center}

\begin{center}
\includegraphics[angle=0,width=5cm]{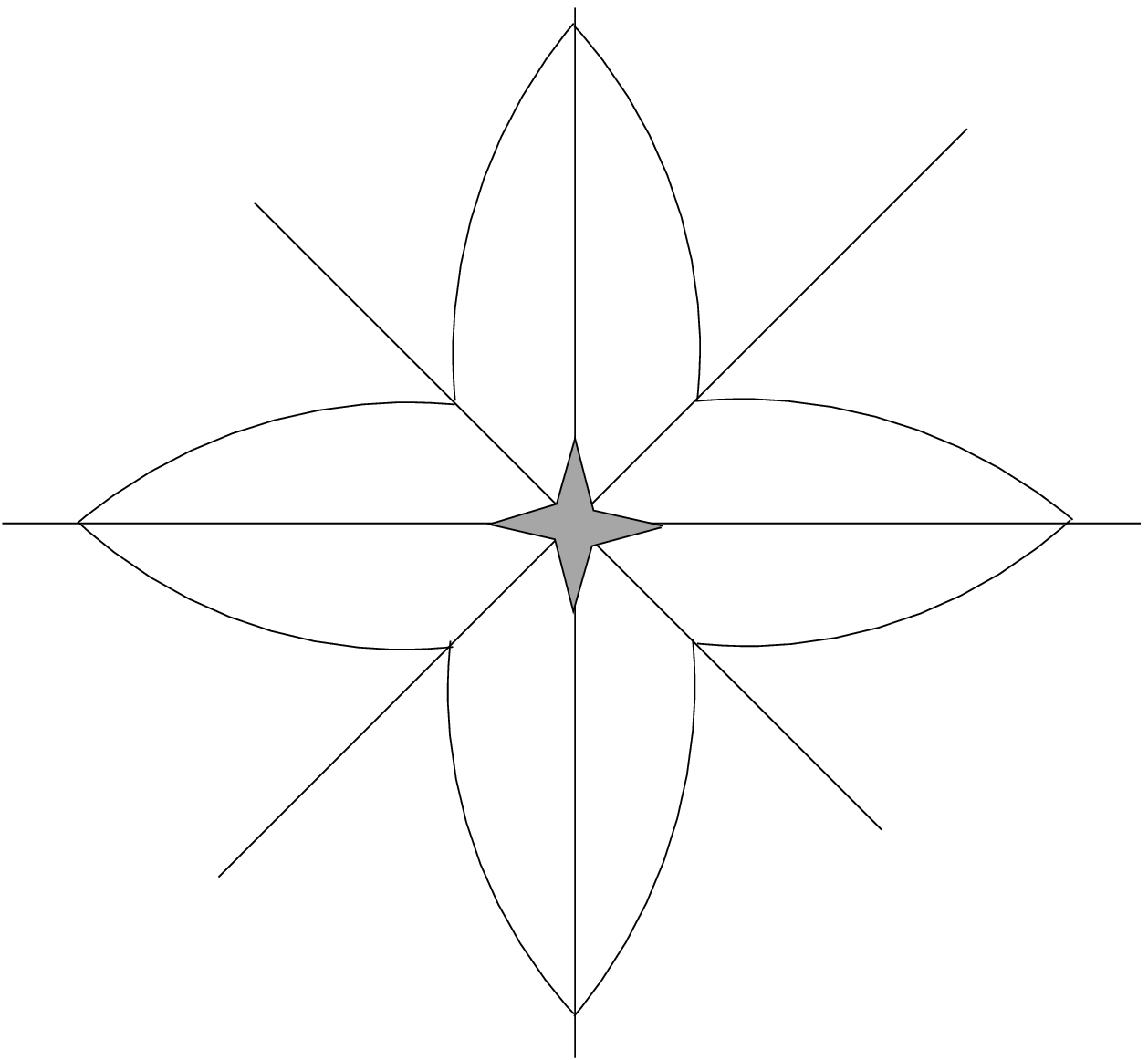}\hspace{1cm}
\includegraphics[angle=0,width=5cm]{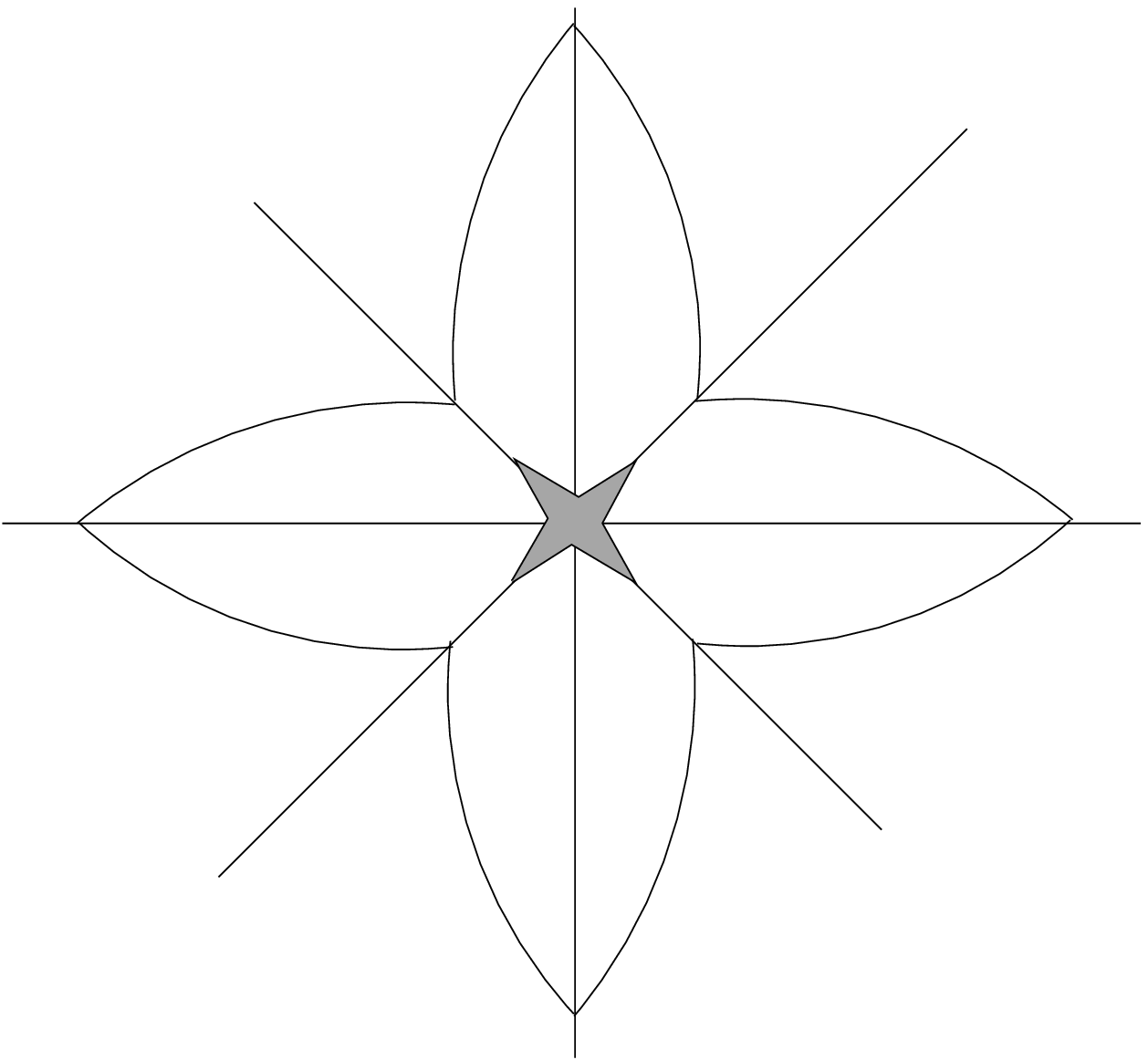}\\
{\small Examples of maximal (left) and minimal (right) configurations\\ with $n=2,$ $3$ and $4$ respectively}
\end{center}

Extensions of Theorem 1 to the following situations can be obtained up to slight changes in the proof (indeed, only the Hadamard formula should be replaced by the variation formula corresponding to the new functional):
\begin{enumerate}
	\item Soft obstacles: instead considering the Dirichlet Laplacian on $D\setminus  B$, we consider  the Schrödinger type operator 
$$H(\alpha,B):=\Delta - \alpha\chi_B$$
acting on $H^1_0(D)$,
where $\alpha>0$ and $\chi_B$ is the indicator function of $B$. Optimization problems related to the fundamental eigenvalue of operators of this kind have been investigated in particular in \cite {HKK} and \cite{ C}. Under the assumptions of Theorem \ref{main} on $D$ and $B$, $\forall \alpha>0$, the fundamental eigenvalue of $H(\alpha,B)$ achieves its maximum at the ``ON'' position and its minimum at the ``OFF" position. 
\item Wells: this case corresponds to the operator $H(\alpha,B)$ with $\alpha<0$. Under the circumstances of Theorem \ref{main}, $\forall \alpha<0$, the first eigenvalue of $H(\alpha,B)$ achieves its maximum at the ``OFF'' position and its minimum at the ``ON" position.  
	\item  Stationary problem : the problem now is to optimize the Dirichlet energy $J(D\setminus B):=\int_{D\setminus B} |\nabla u|^2 dx$ of the unique solution $u$ of the problem
	\begin{displaymath}
\left\{\begin{array}{rcll}
\Delta u&=&- 1 &\textrm{in  } D\setminus B\\
u&=&0  &\textrm{on  }\partial (D\setminus B),
\end{array}\right.
\end{displaymath}
This problem was treated in \cite[Section 2]{K} in the case where both $D$ and $B$ are balls. Under the assumptions of Theorem \ref{main} on $D$ and $B$, one can prove that $J(D\setminus B)$ achieves its maximum when $B$ is at the ``ON'' position and its minimum when $B$ is at the ``OFF" position.
\end{enumerate}

 \section {Proof of the main result}\label{2}
 Without loss of generality, we may assume that the domain $D$ and the obstacle $B$ are centered at the origin and are both symmetric with respect to the $x_1$-axis so that they can be parametrized in polar coordinates by
   $$D=\{re^{\iz\theta}; \theta\in [0,2\pi),0\leq r< g(\theta)\},$$
   $$B=\{re^{\iz\theta}; \theta\in [0,2\pi),0\leq r< f(\theta)\},$$
   where $f$ and $g$ are two 
 positive even $\frac{2 \pi}n$-periodic functions which are \emph{nondecreasing} on  $(0,\frac\pi n)$. To avoid technicalities, we suppose throughout that $g$ is continuous and $f$ is $C^2$.  Extensions of our result to a wider class of domains would certainly be possible up to some additional technical difficulties.

 The condition that  the obstacle $B$ can freely rotate around his center inside $D$, that is $\rho (\bar B)\subset D$ for all $\rho\in SO(2)$, amounts to the following:
 $$f(\frac\pi n)=\max_{0\le\theta\le 2\pi} f(\theta)< \min_{0\le\theta\le 2\pi} g(\theta)=g(0).$$
 
Let us denote, for all $t\in\R$, by $\rho_t$ the rotation of angle $t$, that is, $\forall \zeta\in\R^2\cong \C$,  $\rho_t(\zeta)=e^{\iz t}\zeta$, and set 
\begin{center}
$B_t:=\rho_t(B)$ and $\Omega(t):=D\setminus B_t$.
\end{center}
 Let $\lambda(t)$ be the fundamental eigenvalue of the 
Dirichlet Laplacian on $\Omega(t)$. It is well known that, since it is simple, the first Dirichlet eigenvalue $\lambda(t)$ is a differentiable function of $t$ (see \cite{GS, Re} ). We denote by $u(t)$ the one parameter family of nonnegative first eigenfunctions satisfying, $\forall t\in\R$,
\begin{displaymath}
\left\{\begin{array}{rcll}
\Delta u(t)&=&- \lambda(t) u(t) &\textrm{in  } \Omega(t)\\
u(t)&=&0  &\textrm{on  }\partial\Omega(t)\\
\int_{\Omega(t)} u^2(t)&=&1.
\end{array}\right.
\end{displaymath}
The derivative of $\lambda(t)$ is then given by the following so-called Hadamard formula (see \cite{EI1, GS, Ha,Sc2}):
\begin{equation}\label{hadamard}
\lambda'(t)= \int_{\partial B_t}
\left|\frac{\partial u(t)}{\partial {\eta_t}}\right|^2 {\eta_t}\cdot{v} \ d\sigma,
\end{equation}
where ${\eta_t}$ is the inward unit normal vector field of $\partial\Omega(t)$ (hence, along $\partial B_t$ the vector ${\eta_t}$ is outward with respect to $B_t$) and 
${v}$ denotes the restriction to $\partial\Omega(t)=\partial D\cup \partial B_t$ of the deformation vector field. In our case, the vector ${v}$ vanishes on $\partial D$ and is given by ${v}(\zeta)=\iz \zeta
$ for all $\zeta \in  \partial B_t$. 

Since both $\Omega$ and $B$ are invariant by the dihedral group $\mathbb{D}_n$, it follows that, $\forall t\in\R$,  $\Omega (t+\frac {2\pi} {n})=\Omega_t$.  Moreover, if we denote by $S_0$ the reflection with respect to the $x_1$-axis, then we clearly have $\rho_{-t}=S_0\circ \rho_t\circ S_0$ which gives $B_{-t}= S_0 (B_t)$ and $\Omega_{-t}= S_0 (\Omega_t)$. 
Hence, as a function of $t$, the first Dirichlet eigenvalue of $\Omega_t$ is even and periodic of period $\frac {2\pi} {n}$, that is, $\forall t\in\R$,
\begin{center}
$\lambda(t+\frac {2\pi} {n})=\lambda(t)$ and $\lambda(-t)=\lambda(t)$.
\end{center}
 Therefore, it suffices to investigate the variations of $\lambda(t)$ on the interval $\left[0,\frac \pi n\right]$ and  Theorem \ref{main} is a consequence of the following:

\begin{theorem}\label{mainth}
Assume that neither $D$ nor $B$ is a disk. 
\begin{itemize}
	\item[(i)] $\forall t\in \left(0,\frac\pi n\right)$, $\lambda'(t)<0$. Hence, $\lambda(t)$ is strictly decreasing on  $\left(0,\frac \pi n\right)$.
	\item[(ii)]  $\forall k\in\Z$, $\lambda'(k\frac\pi n)=0$ and $k\frac\pi n, \; k\in\Z$, are the only critical points of $\lambda$ on $\R$.
\end{itemize}
\end{theorem}

Hence, $\lambda(t)$ achieves its maximum for $t=0 \mod \frac{2\pi}{n}$ which corresponds to the ``ON'' position,  and its minimum for $t=\frac{\pi}{n} \mod \frac{2\pi}{n}$ which corresponds to the ``OFF'' position. Of course, if $D$ or $B$ is a disk, then the function $\lambda(t)$ is constant.

In what follows we will denote, for any $\alpha\in\R$, by $z_\alpha$ the $\theta=\alpha$ axis, that is $z_\alpha:=\{re^{\iz\alpha};\; r\in\R\}$, and by $z^+_\alpha$ the half-axis $\{re^{\iz\alpha};\; r\geq 0\}$.

We start the proof with the following elementary lemma.
\begin{lemma}\label{n.v}
Let $K$ be a plane domain defined in polar coordinates by 
$K=\{re^{\iz\theta}; \theta\in [0,2\pi),0\leq r< h(\theta)\},$
where $h$ is a positive $2\pi$-periodic function of classe $C^1$, and let $v$ be a vector field whose restriction to $\partial K$ is given by
$$v(\theta):={v}(h(\theta)e^{\iz\theta})=\iz h(\theta)e^{\iz\theta}
=h(\theta)e^{\iz(\theta+\frac\pi 2)}.$$ 
We denote by $\eta$ the unit outward normal vector field of $\partial K$. One has, at any point $h(\theta)e^{\iz\theta}$ of $\partial K$ where $\eta$ is defined,  
\begin{enumerate}
\item[(i)] 
$\eta(\theta):={\eta}(h(\theta)e^{\iz\theta})=\frac{h(\theta)e^{\iz\theta}-\iz h'(\theta)e^{\iz\theta}}
{\sqrt{h^2(\theta)+h'^2(\theta)}}$ 
\item[(ii)] ${\eta}\cdot{v}(\theta)=
      \frac{-h(\theta)h'(\theta)}{\sqrt{h^2(\theta)+h'^2(\theta)}}$. Hence, ${\eta}.{v}(\theta)$  has constant sign on an interval 
$I$ if and only if $h$ is monotonous in $I$.
\item[(iii)] if for some $\alpha>0$, the domain $K$ is symmetric with respect to the axis $z_\alpha$, then the function ${\eta}\cdot{v}$ is antisymmetric w.r.t this axis, that is
     $${\eta}\cdot{v}(\alpha+\theta )=-{\eta}\cdot{v}(\alpha -\theta).$$
\end{enumerate}
\end{lemma}

\begin{proof}
Assertions (i) and (ii) are direct consequences from the definition of $K$. The fact that $K$ is  symmetric with respect to the axis $z_\alpha$ implies that the function $h$ satisfies $h(\alpha+\theta)=h(\alpha -\theta)$. Therefore, (iii) follows immediately from (ii).
\end{proof}

We will denote by $S_\alpha$ 
the symmetry with respect to the axis $z_\alpha$. We will also denote, for $\alpha <\beta$,  by $\sigma\left(\alpha,\beta\right)$ the sector delimited by $z^+_\alpha$ and $z^+_\beta$, that is
$$\sigma\left(\alpha,\beta\right)=\{r e^{\iz\theta}; r> 0\;\mbox{and} \; \alpha<\theta<\beta\}.$$

\begin{lemma}\label{prosymint}
Let $D$ be as above. For all $t\in\left(0,\frac\pi n\right)$, we have:
$$S_{\frac\pi n+t}\left(D\cap\sigma
\left(\frac\pi n+t,\frac{2\pi}n+t\right)\right)\subseteq D\cap\sigma\left( t,\frac\pi n+t\right).$$
Moreover,  if $D$ is not a disk, then $$S_{\frac\pi n+t}\left(\partial D\cap\sigma
\left(\frac\pi n+t,\frac{2\pi}n+t\right)\right)\cap D\neq \emptyset.$$
\end{lemma}

\begin{proof}

The action of the symmetry $S_{\frac\pi n+t}$ is given in polar coordinates by $S_{\frac\pi n+t}(re^{\iz\theta})=r e^{\iz(2(\frac\pi n+t)-\theta)}$. Hence, 
$$S_{\frac\pi n+t}\left(D\cap\sigma
\left(\frac\pi n+t,\frac{2\pi}n+t\right)\right)=S_{\frac\pi n+t}(D)\cap\sigma\left(t,\frac\pi n+t\right).$$  
Moreover, the domain $D$ being parametrized by a positive even $\frac{2 \pi}n$-periodic function $g(\theta)$, that is $D=\{re^{\iz\theta}; \theta\in [0,2\pi),0\leq r< g(\theta)\},$ its image $S_{\frac\pi n+t}\left(D\right)$ can be parametrized in the same manner by the function $g^*(\theta)=g(\theta - 2t)$. Thus
$$S_{\frac\pi n+t}(D)\cap\sigma\left(t,\frac\pi n+t\right)=\{re^{\iz\theta}; \theta\in \left(t,\frac\pi n+t\right),0\leq r< g(\theta - 2t)\}.$$ 
Therefore, we need to prove that  $F(\theta)=g(\theta)-g^*(\theta)$ is  nonnegative for every $\theta$ in the interval $(t,\frac\pi n+t)$. This will be possible thanks to the assumptions of symmetry (that is $g$ is even and $\frac{2 \pi}n$-periodic) and monotonicity (that is $g$ is nondecreasing on $[0,\frac\pi n]$). Indeed, these properties imply that on the interval $\left(t,\frac\pi n+t\right)$,
\begin{itemize}
	\item $g$ achieves its maximum at $\theta=\frac{\pi}n$,
	\item $g^*$ achieves its minimum at $\theta=2t$.
\end{itemize}
\begin{center}
\includegraphics[angle=0,width=5cm]{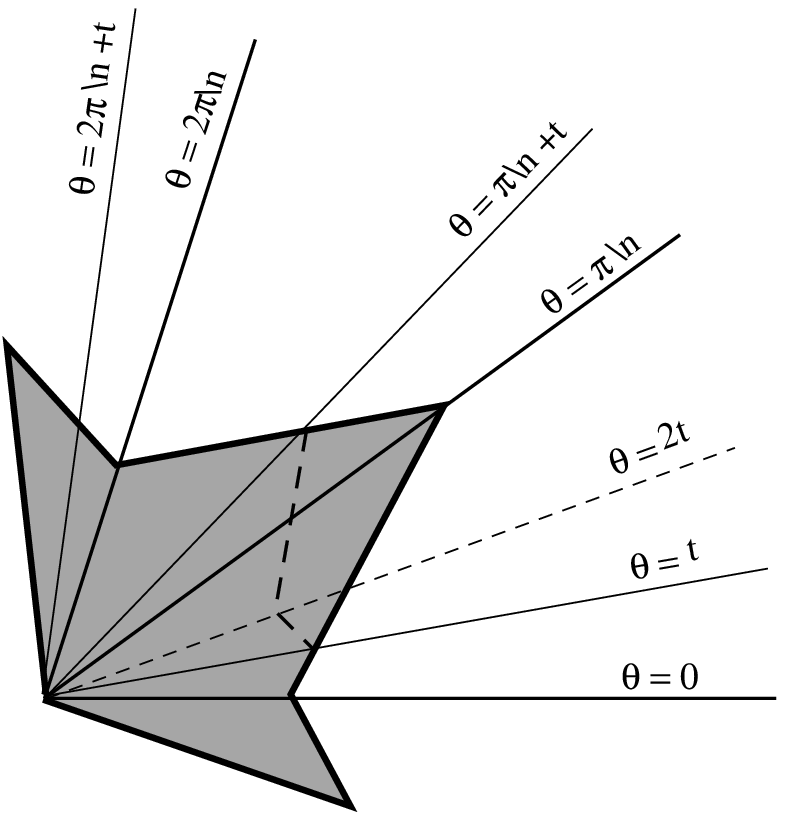}
\hfill
\includegraphics[angle=0,width=5cm]{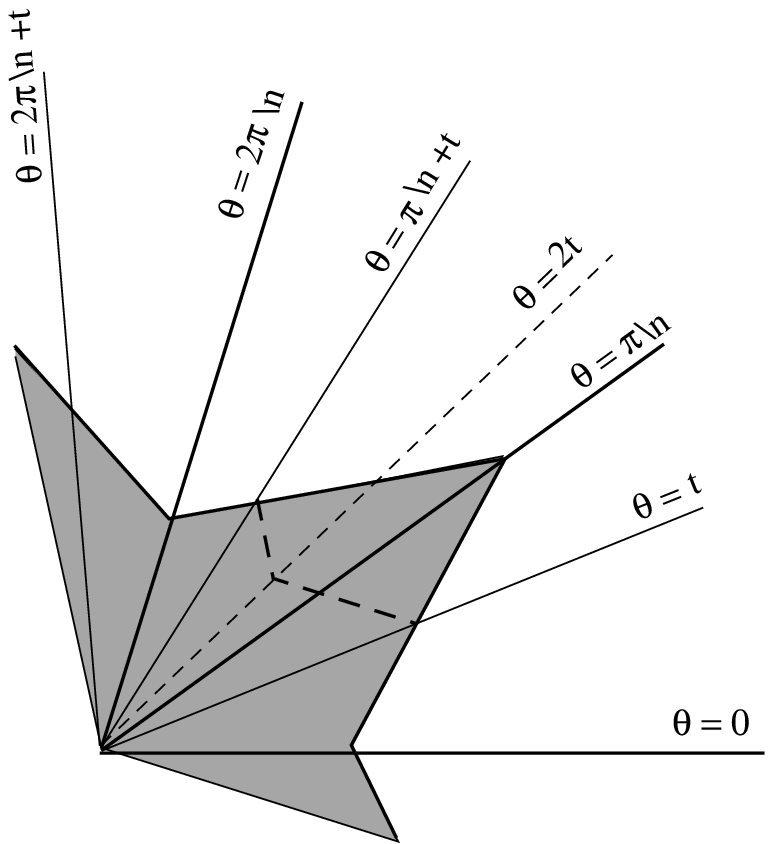}
\end{center}
\begin{center}
 \small case $2t<{\pi\over n}$ \hspace{6cm} case $2t>{\pi\over n}$
\end{center}
Four cases must be considered separately:
\begin{itemize}
\item[-] If $t<\theta\leq \min\{2t,\frac \pi n\}$, we may write, since $g$ is even,
$F(\theta)=g(\theta)-g(2t -\theta)$, with 
$0\leq 2t-\theta<\theta \leq\frac\pi n$. Since $g$ is nondecreasing on $[0,\frac\pi n]$, we get $F(\theta) \ge 0$.
    
\item[-]  If $\max\{2t,\frac \pi n\}\le \theta< \frac\pi n+t$, we may write, since $g$ is even and $\frac{2 \pi}n$-periodic,  
$F(\theta)=g(2\frac\pi n -\theta)-g(\theta-2t)$
with $0\leq \theta-2t< 2\frac\pi n -\theta \leq\frac\pi n$. Hence, $F(\theta) \ge 0$.
    
\item[-] If $2t <\frac \pi n$ and  $2t\le\theta\le\frac \pi n$, then $0 \leq \theta - 2t< \theta\leq\frac\pi n$ and, then, $F(\theta)=g(\theta)-g(\theta - 2t) \ge0$. 

\item[-] If $2t >\frac \pi n$ and $\frac \pi n\le\theta\le 2t$, then $0\leq 2t-\theta< 2\frac\pi n -\theta \leq\frac\pi n$ and, then, $F(\theta)=g(2\frac\pi n -\theta)-g(2t-\theta) \ge0$. 
\end{itemize}
Hence, $F(\theta)$ is nonnegative for all $\theta$ in $(t,\frac\pi n+t)$. 

Now, if $D$ is not a disk, then $g$ is nonconstant on $[0,\frac\pi n]$. Following the arguments above, we deduce that the function $F(\theta)$ is positive somewhere on $(t,\frac\pi n+t)$ which means that $S_{\frac\pi n+t}\left(\partial D\cap\sigma
\left(\frac\pi n+t,\frac{2\pi}n+t\right)\right)$ meets the interior of $D$.
\end{proof}
\begin{proof}[Proof of Theorem \ref{mainth}] Notice first that, since $\lambda$ is an even and $\frac{2\pi}{n}$-periodic function of $t$, one immediately gets, $\forall k\in \Z$, $\lambda(k\frac\pi n-t)=\lambda(k\frac\pi n+t)$ and, then,
$$\lambda'\left(k\frac\pi n\right)=0.$$
Alternatively, one can deduce that $\lambda'\left(k\frac\pi n\right)=0$ from Hadamard's variation formula (\ref{hadamard}) after noticing that the domain $\Omega(k\frac\pi n)$ is symmetric with respect to the $x_1$-axis and that the first Dirichlet eigenfunction $u(k\frac\pi n)$  satisfies
$u\circ S_0=u$, where $S_0$ is the symmetry with respect to the $x_1$-axis.

Let us fix a $t$ in $\left(0,\frac{\pi}n \right) $ and denote by $u$ the nonnegative first Dirichlet eigenfunction of $\Omega (t)$ satisfying $\int_{\Omega(t)} u^2=1$. The domain $\Omega(t)$ is clearly invariant by the rotation $\rho_{\frac{2\pi}n}$ of angle $\frac{2\pi}n$, hence $u\circ \rho_{\frac{2\pi}n}=u$. On the other hand, the domain $B$ being parametrized by a positive even $\frac{2 \pi}n$-periodic function $f(\theta)$, that is $B=\{re^{\iz\theta}; \theta\in [0,2\pi),0\leq r< f(\theta)\},$ one has
$$B_t=\{re^{\iz\theta}; \theta\in [0,2\pi),0\leq r< h(\theta)\},$$
with $h(\theta)=f(\theta-t).$ Hence, the function ${\eta_t}\cdot{v}$ is invariant by $\rho_{\frac{2\pi}n}$ (Lemma \ref{n.v}) and we have (Hadamard formula (\ref{hadamard})) 
$$\lambda'(t)=\int_{\partial B_t}\left|\frac{\partial u}{\partial {\eta_t}}
\right|^2 {\eta_t}\cdot{v} \ d\sigma
= n\int_{\partial B_t\cap\sigma(t,\frac{2\pi}n+t)}
\left|\frac{\partial u}{\partial {\eta_t}}\right|^2 {\eta_t}\cdot{v} \ d\sigma.$$
Since $B_t$ is symmetric with respect to the axis $z_{\frac\pi n+t}$, we have (Lemma \ref{n.v}),
${\eta_t}\cdot{v}(\frac\pi n+t+\theta)=-{\eta_t}\cdot{v}(\frac\pi n+t-\theta)$ or, equivalently, ${\eta_t}\cdot{v}(x)=-{\eta_t}\cdot{v}(x^*)$, where $x^*$ denotes the symmetric of $x$ with respect to $z_{\frac\pi n+t}$. This yields
$$\lambda'(t)= n\int_{\partial B_t\cap\sigma(\frac{\pi}n+t,\frac{2\pi}n+t )}
\left(
\left|\frac{\partial u}{\partial {\eta_t}}(x)\right|^2 
-\left|\frac{\partial u}{\partial {\eta_t}}(x^*)\right|^2
\right)
{\eta_t}\cdot{v} (x)\ d\sigma$$ 
Notice that the function $h(\theta)$ is decreasing between $\frac{\pi}n+t$ and $\frac{2\pi}n+t$ and, then, ${\eta_t}\cdot{v}$ is nonnegative on  $\partial B_t\cap\sigma(\frac{\pi}n+t,\frac{2\pi}n+t )$ (Lemma \ref{n.v}).

Let $H(t):=\Omega(t)\cap\sigma(\frac\pi n+t,\frac{2\pi}n+t)$. Applying Lemma \ref{prosymint}, and since $B_t$ is symmetric with respect to the axis $z_{\frac\pi n+t}$, one gets 
$$S_{\frac\pi n+t}(H(t))\subset\Omega(t)\cap\sigma( t,\frac\pi n+t).$$
Hence, the function $w(x)=u(x)-u(x^*)$ is well defined on $H(t)$ and satisfies
$w(x)=0$ for all $x$ in $\partial H(t)\cap \left(\partial B_t\cup z_{\frac\pi n+t} \cup z_{\frac{2\pi} n+t}\right)$. Moreover, since $u$ vanishes on $\partial D$ and is positive inside $\Omega(t)$,   $w(x)\le0$ for all $x$ in $\partial H(t)\cap \partial D$ and $w(x)<0$ for certain $x$ in $\partial H(t)\cap \partial D$ (recall that $D$ is not a disk and apply the second part of Lemma  \ref{prosymint}).

Therefore, the nonconstant function $w$ satisfies the following:
\begin{displaymath}
\left\{\begin{array}{rcll}
\Delta w&=&-\lambda(t) w &\textrm{in  } H(t)\\
w&\leq&0  &\textrm{on  }\partial H(t).
\end{array}\right.
\end{displaymath}
Hence, $w$ must be nonpositive on the whole of $H(t)$. Otherwise, a nodal domain $V\subset H(t)$ of $w$ would have the same first Dirichlet eigenvalue as $\Omega (t)$. But, due to the invariance of $\Omega (t)$ by $\rho_{\frac{2\pi}n}$, the domain $\Omega (t)$ would contain $n$ copies of $V$ leading to a strong contradiction with the domain monotonicity theorem for eigenvalues.  Therefore, $\Delta w \ge 0$ in $H(t)$ and $w$ achieves its maximal value (i.e. zero) on $\partial B_t\cap\sigma(\frac\pi n+t,\frac{2\pi}n+t)\subset \partial H(t)$. The Hopf maximum principle (see \cite[Theorem 7, ch.2]{PW}) then implies that, at any regular point $x$ of $\partial B_t\cap\sigma(\frac\pi n+t,\frac{2\pi}n+t)$, one has 
$$\frac{\partial w}{\partial {\eta_t}}(x)=\frac{\partial u}{\partial {\eta_t}}(x)-\frac{\partial u}{\partial {\eta_t}}(x^*)<0.$$
It follows that $\lambda'(t)\le 0$ and that the equality holds if and only if ${\eta_t}\cdot{v}\equiv0$.  By Lemma \ref{n.v}, this last equality occurs if and only if $f$ is constant which means that $B$ is a disk.

\end{proof}

\end{document}